\documentclass[10pt]{amsart}
\usepackage{amssymb,amsthm,amsmath,amsfonts}
\usepackage{hyperref}
\usepackage{enumerate}

\makeatletter
\g@addto@macro\th@plain{\thm@headpunct{}}
\makeatother

\title{A renewal theorem and supremum of a perturbed random walk}
\author[E. Damek]{Ewa Damek}
\address{Institute of Mathematics\\ Wroclaw University\\ 50-384 Wroclaw\\
	pl. Grunwaldzki 2/4\\ Poland}
\email{edamek@math.uni.wroc.pl}

\author[B. Ko{\l}odziejek]{Bartosz Ko{\l}odziejek}
\address{Faculty of Mathematics and Information Science\\Warsaw University of Technology\\ Koszykowa 75\\00-662 Warsaw, Poland}
\email{b.kolodziejek@mini.pw.edu.pl}



\newtheorem{theorem}{Theorem}[section]
\newtheorem{proposition}[theorem]{Proposition}

\newtheorem{remark}[theorem]{Remark}

\newcommand{\dd}{\mathrm{d}}
\newcommand{\Ind}{\mathbf{1}_}
\newcommand{\eps}{\varepsilon}
\renewcommand{\a}{\alpha}
\renewcommand{\b}{\beta}
\renewcommand{\d}{\delta}
\newcommand{\E}{\mathbb{E}}
\newcommand{\N}{\mathbb{N}}
\newcommand{\Z}{\mathbb{Z}}
\renewcommand{\P}{\mathbb{P}}
\newcommand{\R} {{\mathbb R}}

\begin{document}
	\begin{abstract}{We study tails of the supremum of a perturbed random walk under regime which was not yet considered in the literature. Our approach is based on a new renewal theorem, which is of independent interest.
			
			We obtain first and second order asymptotics of the solution to renewal equation under weak assumptions and we apply these results to obtain first and second order asymptotics of the tail of the supremum of a perturbed random walk.}
	\end{abstract}
\maketitle

\section{Introduction}
\subsection{Renewal theorems} Almost every renewal quantity may be described as the solution $f$ to an integral equation
\begin{equation}\label{reneq}
f=\psi + f\ast\mu,
\end{equation}
where $\mu$ is a probability measure and $\psi$ is a locally bounded function.
When functions $f$, $\psi$ and measure $\mu$ are supported on $\R^+$, then $f$ is given by 
\begin{equation}\label{rensol}
f(x)=\int _{\R }\psi (x-y ) H(\dd y)=\int _{(0,x]} \psi (x-y ) H(\dd y),
\end{equation}
where $H=\sum _{n=0}^{\infty }\mu^n$, provided \mbox{$\lim_{n\to \infty }\psi \ast\mu^n=0$}. Such equations appear frequently in different problems. In particular, they are closely related to stochastic fixed point equations, {\cite{Gol91}}.

In the general case when $\psi$ and $\mu$ are {defined} on $\R$ and $\mu$ has a strictly positive mean, the second equality in \eqref{rensol} does not hold but quite likely, due to properties of $\psi $ and $\mu$,
\[
\int _{(0,x]} \psi (x-y) H(\dd y)
\]
may become the main term in the asymptotics {of $f(x)$} as $x \to \infty $. The classical Key Renewal Theorem {(KRT)} gives the asymptotics of $f$ at infinity when $\psi $ is directly Riemann integrable but {in applications one encounters} equations \eqref{reneq} where {corresponding} $\psi$ is not even in $L^1$ and so there is a need of a more general theory.

In this paper we study behavior of the integrals
\begin{align}\label{eqLL}
\int _{(0,x]} L\left( e^{x-z}\right) H(\dd z)\qquad\mbox{ and }\qquad\int_{\R} L\left( e^{x-z}\right) H(\dd z),\quad  \mbox{as}\ x\to\infty,
\end{align}
where $L$ is a slowly varying function with the property that $L(e^x)$ is not integrable on $\R^+$. 
The asymptotics of \eqref{eqLL} does not follow from the classical KRT nor from known relatives (see e.g.  \cite[Section 6.2.3]{Iks}), 
because the latter results are obtained under additional assumption that the integrand function is (ultimately) monotone or is asymptotically equivalent to a monotonic function. In our case, $x\mapsto L\left(e^x\right)$ may exhibit infinite oscillations, so in general it is not asymptotically equivalent to a monotonic function. Slow
variation of $L$ plays here the role of the regularity condition. For $0\leq x_0 <x$, let
\begin{align}\label{defL}
\widetilde{L}(x_0,x)=\int_{x_0}^xt^{-1}L(t)\dd t.
\end{align}
{If the integral exists}, $\widetilde{L}(x_0,x)$ is again slowly varying as a function of $x$. Under very mild assumptions on $L$, we show that (Theorem \ref{corL})
\begin{align*}
\int_{(0,x]} L\left( e^{x-z}\right) H(\dd z)\sim \frac{1}{m}\widetilde{L} (x_0,e^x)\qquad\mbox{ as }x\to\infty
\end{align*}
for any $x_0>0${, where $m$ is the mean of $\mu$}.
Here and henceforth, $f(x)\sim g(x)$ means that $f(x)/g(x)\to 1$ as $x\to\infty$.
The proof of Theorem \ref{corL} is surprisingly simple and although we use particular properties of slowly varying functions, the scheme behind it may be adopted to other situations.

Under some further assumptions we are able to reduce the second quantity in 
\eqref{eqLL} to the first one and to obtain
the second order asymptotics (Theorem \ref{LTH}). Imposing some more regularity on {$\mu$}, we show that
\begin{align*}
\int_{\R} L\left( e^{x-z}\right) H(\dd z)= \frac{1}{m}\widetilde{L}(0,e^x)
+O\left(L\left(e^x\right)\right),\qquad\mbox{ as }x\to\infty.
\end{align*}
\subsection{Supremum of a perturbed random walk}
Consider
\begin{align}\label{defR}
R=\sup_{n\geq 1}\left\{A_1\cdot\ldots\cdot A_{n-1}\cdot B_n\right\},
\end{align}
where the sequence $(A_n,B_n)_{n\geq 1}$ is i.i.d.  and $A_1\geq 0$ a.s. 
The renewal {equation \eqref{reneq}} arises naturally from the study of the {right} tail of $R$. When $A_1$ and $B_1$ are positive a.s. then taking the logarithm of both sides of \eqref{defR}, $M=\log R$ is the supremum of the so-called perturbed random walk (PRW)
\begin{equation}\label{PRW}
X_1+\dots +X_{n-1}+\xi _n,
\end{equation}
where $X_1=\log A_1$, $\xi _1=\log B_1$. In a more general case 
considered here it is natural to call $A_1\cdot\ldots\cdot A_{n-1}\cdot B_n$ a  
perturbed multiplicative random walk.
Random variable $R$, if exists, satisfies the following fixed-point stochastic equation
\begin{align*}
R\stackrel{d}{=}\max \{ AR,B\},  \qquad R\mbox{ and }(A,B)\mbox{ on the r.h.s. are independent}.
\end{align*} 
$\P(R>x)$ converges to zero when $x$ tends to infinity and it is a natural question to describe the rate at which it happens. We do it here under specific assumptions and the result clearly applies to $M$. PRW is a natural extension of the random walk with applications to queuing theory, insurance risk as well as to telecommunication networks \cite{Iks}, \cite{Palm07}, \cite{Schle98}, \cite{Schm95}. 
Supremum of the process with regenerative increments can be represented as $M$ for an appropriate PRW and the size of the largest box in the Bernoulli sieve as $R$, \cite{Iks}. 
Therefore, supremum of PRW is {a natural object }
to study. 
Supremum of PRW, $M$, inherits characteristics related to $\max _n \{X_1+\dots + X_n\} $ and to the extreme behavior of the perturbations $\xi _n$. Tail behavior of $M$ is studied  in three main regimes in \cite{Palm07}, one of them being the Cramer case which 
requires existence of some exponential moments of $X_1$ and $\xi _1$.    
Since our framework is a little bit more general, we shall explain it in terms of $A_1=e^{X_1}, B_1=e^{\xi _1}$.

In the Cramer case the tail behavior of $R$ may be determined by $A$ or $B$ alone, or by both of them. The first case happens when the tail of $B$ is regularly varying with index $-\a<0$, $\E A^{\a }<1$ and $\E A^{\a+\eps}<\infty$ for some $\eps>0$. Then (\cite[Theorem 3]{Ara06}, \cite[Theorem 2]{Palm07}) 
\begin{align*}
\P (R>x)\sim \frac{1}{1-\E A^\a}\P (B>x).
\end{align*}
On the other hand, if $\E A^{\a }=1$, $\E B_+^{\a }<\infty$, $\rho=\E A^{a }\log A<\infty$ and the distribution of $\log A$ given $A>0$ is non-arithmetic, then (\cite[Theorem 5.2]{Gol91}) 
\begin{align*}
\lim_{x\to\infty}x^\a\P (R>x)=\frac{1}{\alpha\rho}\E \left(\max(AR_+,B_+)^\a -(AR_+)^\a\right)
\end{align*}
and it is $A$ that plays the main role here.

The Cramer case when both $A$ and $B$ contribute significantly to the tail has not been considered yet for $R$ although it is quite natural to do it, see the example described in Subsection \ref{extremes}.

We assume that 
\begin{align*}
&\mbox{there exists }\alpha>0 \mbox{ such that }\E A^{\a }=1,\qquad \rho=\E A^{\a }\log A<\infty, \\
&B \,\, \mbox{has a regularly varying right tail with index} \ -\a,\qquad \E B_+^{\a }=\infty  
\end{align*}
(see Section \ref{taila} for the rest of assumptions) and we describe the right tail of $R$ (Theorem \ref{pert}). If \mbox{$\P(B>x)=x^{-\alpha}L(x)$} for $x>0$, where $L$ is a slowly varying function, then 
\begin{equation}\label{mainass}
x^{\a}\P(R>x)\sim \frac{1}{\rho }\widetilde{L}(0,x) \sim\, \frac{\E B_+^\alpha\Ind{B\leq x}}{\alpha\rho}\to\infty.
\end{equation} 
and so the tail is essentially bigger than that of $B${, since $\widetilde{L}(0,x)/L(x)\to\infty$ as $x\to\infty$; see \eqref{Lreg2}}. Appearance of the function $\widetilde{L}$ is probably the most interesting phenomenon here.
To obtain \eqref{mainass} we use a renewal theorems mentioned in the previous section. {A somehow related problem was dealt with in \cite{Kevei2016}.}

Finally, assuming some more regularity on the distribution of $A$, we obtain the second order asymptotics in \eqref{mainass}, that is,
\begin{equation*}
x^{\a }\P (R>x)=\frac{1}{\rho }\widetilde{L}(0,x) -\frac{\E\min\{AR,B\}_+^\alpha}{\alpha\rho} + O( L(x) ),\qquad\mbox{as }x\to\infty,
\end{equation*}
see Theorem \ref{pert} {\rm (ii)}. Note that if $L$ is asymptotically bounded away from zero, then we have $x^{\a }\P (R>x)=\rho^{-1}\widetilde{L}(0,x)+O(L(x))$, but such claim is not true if $L$ is decreasing to $0$ (e.g. $L(x)\sim (\log x)^{-1}$).

We hope that \eqref{mainass} holds in a more general setting: 
\[R\stackrel{d}{=}\psi(R),\qquad R\mbox{ and }\psi\,\mbox{ are independent},\]
where $\psi\colon \R\to \R$ is a random {Lipschitz} mapping such that for any $x\in \R$
\[\max(Ax,B+C_1)\leq \psi(x)\leq Ax+B+C_2,\qquad \mbox{a.s.}\]
and $\E C_i^\alpha<\infty$. 
But then calculations become much more technical and it is still a work in progress.

\subsection{Extremes of perturbed random walk}\label{extremes} There is a somehow related problem, where contributions to asymptotics of some statistic may come from one of two ingredients alone or from both of them. 
Let $(\xi_n,\eta_n)_{n}$ be a sequence of i.i.d. two-dimensional random vectors with generic copy $(\xi,\eta)$. Consider the maximum of PRW, $M_n=\max_{1\leq k\leq n}\{ S_{k-1}+\eta_k\}$, where $(S_n)_{n\geq 1}$ is a random walk with i.i.d. increments $\xi_k$, $\E\xi_k=0$ and $\E\xi_k^2<\infty$, $S_0=0$. The aim is to study convergence in distribution of $a_n M_n$ for some suitable chosen deterministic sequence $(a_n)_{n\geq 1}$. There are essentially three distinct cases. 
In the first case $\E \eta^2<\infty$, $S_n$ dominates the perturbation and the limit of $a_n M_n$ coincides with the limit of $a_n \max_{1\leq k\leq n}\{ S_{k-1}\}$. In the second one, the tail $\P(\eta>x)$ is regularly varying with index $\gamma\in(-2,0)$, perturbation $\eta_n$ dominates the random walk and the limit coincides with the limit of $a_n\max_{1\leq k\leq n}\{ \eta_k\}$. For above see \cite[Theorem 3]{Ren11}.
In the most interesting, third case, that is, if $\P(\eta>x)\sim c x^{-2}$ for some $c>0$, both random walk and the perturbation have comparable contributions, see \cite{Wan14,IksPi14} along with generalization to functional limit theorems. Further developments have been made recently in \cite{IksPiSa17}, where the assumption $\E\xi_k^2 < \infty$ is dispensed with.
\section{Preliminaries}\label{Preli}

\subsection{Regular variation}
A measurable function $f\colon(0,\infty)\to(0,\infty)$ is called \emph{regularly varying with index $\rho$}, 
$|\rho|<\infty$, if for all $\lambda>0$,
\begin{align}\label{reg}
\lim_{x\to\infty}\frac{f(\lambda x)}{f(x)}=\lambda^\rho.
\end{align}
The class of such functions will be denoted $R(\rho)$.
If $f\in R(0)$ then $f$ is called a \emph{slowly varying function}. The class of slowly varying functions plays a fundamental part in the Karamata's theory of regular variability, since if $f\in R(\rho)$, then $f(x)=x^\rho L(x)$ for some $L\in R(0)$. Below, we introduce some basic properties of the class $R(0)$ that, later on, will be essential. 

If $L\in R(0)$ is bounded away from $0$ and $\infty$ on every compact subset of $[0,\infty)$, then for any $\delta>0$ there exists $A=A(\delta)>1$ such that (Potter's Theorem, see e.g \cite{BDM2016}, Appendix B)
$$\frac{L(y)}{L(x)}\leq A \max\left\{\left(y/x\right)^{\delta},\left(y/x\right)^{-\delta}\right\},\qquad x,y>0.$$

Assume that $L\in R(0)$ is locally bounded on $(X,\infty)$ for some $X>0$. Then, for $\alpha>0$ one has
\begin{align}\label{Lreg}
\int_X^x t^\alpha \frac{L(t)}{t}\mathrm{d}t\sim \alpha^{-1}x^\alpha L(x)
\intertext{and this result remains true also for $\alpha=0$ in the sense that}
\frac{\int_X^x \frac{L(t)}{t}\mathrm{d}t}{L(x)}\to\infty\qquad\mbox{ as }x\to\infty.\label{Lreg2}
\end{align}

Define $\widetilde{L}(x):=\int_X^x t^{-1}L(t)\mathrm{d}t$.
Then, for any $\lambda>0$,
\begin{align}\label{smallint}
\frac{\widetilde{L}(\lambda x)-\widetilde{L}(x)}{L(x)}=\int_1^{\lambda}\frac{L(xt)}{L(x)}\frac{\dd t}{t}\to\log\lambda,
\end{align}
since the convergence in \eqref{reg} is locally uniform outside zero \cite[Theorem 1.5.2]{BGT89}.
Moreover, since
$$\frac{\widetilde{L}(x)}{L(x)}\left( \frac{\widetilde{L}(\lambda x)}{\widetilde{L}(x)}-1\right)\to \log\lambda\qquad\mbox{ as }x\to\infty,$$
\eqref{Lreg2} implies that $\widetilde{L}$ is slowly varying.
In the theory of regular variation, $\widetilde{L}$ is called the de Haan function.

\subsection{Renewal theory}
Let $(Z_k)_{k\geq 1}$ be the sequence of independent copies of a random variable $Z$ with $\E Z>0$. 
We write $S_n=Z_1+\ldots+Z_n$ for $n\in\mathbb{N}$ and $S_0=0$.
The measure defined by 
\[H(B):=\sum_{n=0}^\infty \P(S_n\in B),\qquad B\in\mathcal{B}(\R)\]
is called the \emph{renewal measure of $(S_n)_{n\geq 1}$}. 
 Condition $\E Z>0$ along with $\E Z_-^2<\infty$ imply that $H((-\infty,x])$ is finite for all $x\in\R$ (see e.g. \cite[Theorem 2.1]{KM96}).

We say that the distribution of $Z$ is \emph{arithmetic} if its support is contained in $d\Z$ for some $d>0$ and \emph{non-arithmetic}, otherwise. Equivalently, the distribution of $Z$ is arithmetic if and only if there exists $0\neq t\in\R$ such that $f_Z(t)=1$, where $f_Z$ is the characteristic function of the distribution of $Z$.
The law of $Z$ is \emph{strongly non-lattice} if the Cramer condition is satisfied, that is, $\limsup_{|t|\to\infty}|f_Z(t)|<1$.

A fundamental result of renewal theory is the Blackwell theorem \cite{Black53}: if the distribution of $Z$ is non-arithmetic, then for any $t>0$, 
\[H\left((x,x+t] \right)\to \frac{t}{\E Z}\qquad \mbox{ as }x\to\infty.\]

Note that in the non-arithmetic case, since $H\left((x,x+t]\right)$ is convergent as $x\to\infty$ we have $C=\sup_{x}H\left((x,x+1] \right)<\infty$ and so
\begin{align}\label{Hineq1}
H\left((x,x+h] \right)\leq \left\lceil h\right\rceil C\leq \alpha h+\beta
\end{align}
for some positive $\alpha$, $\beta$ and any $h>0$.

Under additional assumptions we know more about the asymptotic behavior of $H$ (see \cite{Stone65}).
If for some $r>0$ one has $\P(Z\leq x)=o(e^{rx})$ as $x\to-\infty$, then there exists $r_1>0$ such that 
\begin{align}\label{Hleft}
H\left((-\infty,x]\right)=o(e^{r_1 x})\qquad\mbox{ as }x\to-\infty.
\end{align}
Exact asymptotics of $H\left((-\infty,x]\right)$ as $x\to-\infty$ in the presence of $\a>0$ such that $\E e^{-\a Z}=1$ are given in \cite{BK16}.

Finally, if $Z$ has finite second moment and for some $r>0$, $\P(Z>x)=o(e^{-r x})$ as $x\to\infty$ and the distribution of $Z$ is strongly non-lattice, then for some $r_1>0$ (see \cite{Stone65})
\begin{align}\label{Hright}
H\left((-\infty,x]\right)=\frac{x}{\E Z}+\frac{\E Z^2}{2(\E Z)^2}+o(e^{-r_1} x)\qquad \mbox{ as }x\to\infty.
\end{align}

\section{Renewal Theorems}
A function $f\colon\R\to\R^+$ is called \emph{directly Riemann integrable} on $\R$ (dRi) if for any $h>0$,
\begin{align}\label{dRi}
\sum_{n\in\Z}\sup_{(n-1)h\leq y<nh}f(y)<\infty
\end{align}
and
\[\lim_{h\to0^+}h\cdot\left(\sum_{n\in\Z}\sup_{(n-1)h\leq y<nh}f(y)-\sum_{n\in\Z}\inf_{(n-1)h\leq y<nh}f(y) \right)=0.\]
If $f$ is locally bounded and a.e. continuous on $\R$, then an elementary calculation shows that \eqref{dRi} with $h=1$ implies direct integrability of $f$.
If the distribution of $Z$ is non-arithmetic, for directly Riemann integrable function $f$, we have the following {\it Key Renewal Theorem}: (see e.g. \cite[Theorem 4.2]{AMN78})
\[\int_{\R} f(x-z)H(\dd z)\to\frac{1}{\E Z}\int_{\R} f(t)\dd t,\qquad x\to\infty.\]
There are many variants of this theorem, when $f$ is not necessarily $L^1$ - see \cite[Section 6.2.3]{Iks}. Such results are usually obtained by additional requirement that $f$ is (ultimately) monotone or $f$ is asymptotically equivalent to a monotone function. 

Here we obtain a renewal result that is essentially stronger: an asymptotic of
\[\int_{(0,x]} L(e^{x-z})H(\dd z)\]
for a slowly varying function $L$, Theorem \ref{corL}. Such a function may exhibit infinite oscillations, so in general it is not asymptotically equivalent to a monotonic function. 

\begin{theorem}\label{corL}
	Assume that $0<\E Z<\infty$ and the law of $Z$ is non-arithmetic.
		Let $L$ be a slowly varying function, which is locally bounded on $[1,\infty)$. For any $x_0\geq 0$ such that  
		 $\widetilde{L}(x_0,x)$
		is finite and $\widetilde{L}(x_0,x)\to\infty$ as $x\to\infty$, one has
		\[\int_{(0,x]} L\left(e^{x-z}\right)H(\dd z)\sim \frac{1}{\E Z}\widetilde{L}\left(x_0,e^x\right).\]
\end{theorem}

\begin{remark}\label{rem}
The assumption that the law of $Z$ is non-arithmetic is not crucial here. The same result holds if one assumes that the law of $Z$ is arithmetic and the proof of such result requires only small modifications necessitated by the use of the Blackwell theorem in the arithmetic case.
\end{remark}
Under stronger assumptions, particularly assuming that $x\mapsto x^{-\alpha}L(x)$ is a monotonic function, we may prove second order asymptotics. 

\begin{theorem}\label{LTH}
	Assume that $0<\E Z<\infty$, $\E e^{\varepsilon Z}<\infty$ for some $\varepsilon>0$, the law of $Z$ is strongly non-lattice and $\P(Z\leq x)=o(e^{rx})$ as $x\to-\infty$ for some $r>0$. Assume further that there is a random variable $B$, a slowly varying function $L$ and a constant $\alpha>0$ such that $\P(B>x)=x^{-\alpha}L(x)$ for $x>0$.
		Let $\widetilde{L}(0,x)\to\infty$ as $x\to\infty$. 
		Then
		\begin{align}\label{eq2}\int_{\R} L\left(e^{x-z}\right)H(\dd z)=\frac{1}{\E Z}\widetilde{L}\left(0,e^{x}\right)+O\left(L(e^x)\right).\end{align}
\end{theorem}

Proofs of both renewal theorems are postponed to Section \ref{proofs}. {We note only that for any slowly varying function there exist $\alpha>0$ and another slowly varying function $L_0$ such that $L(x)\sim L_0(x)$ and $x^{-\alpha}L_0(x)$ is the tail of a probability distribution. In this sense, the assumption of existence $B$ in Theorem \ref{LTH} is not very restrictive.}

\section{Tails of the supremum of perturbed random walk}\label{taila}
\subsection{Notation and assumptions} 
Throughout the paper, $\log$ stands for the natural logarithm. We are going to write $a_+$ for $\max\{a,0\}$.
For any $n\geq 1$ we write $\Pi_n=A_1\cdot\ldots\cdot A_n$ and $\Pi_0=1$.
Our standing assumptions are:
\begin{itemize}
	\item[(A-1)] $\P(A\geq 0)=1$, the law of $\log A$ given $A>0$ is non-arithmetic,
	\item[(A-2)] there exists $\alpha>0$ such that $\E A^{\alpha}=1$, $\E A^{\alpha}\log A<\infty$,
	\item[(B-1)] $L(x):=x^\alpha\P(B>x)\in R(0)$,
	\item[(B-2)] $\E B_+^\alpha=\infty$.
\end{itemize}

\noindent Note that under \rm{(A-2)}
\[\rho=\E A^{\alpha}\log A\]
is strictly positive.
 Indeed, consider $f(\beta):=\E A^\b$. Since $f(0)=1=f(\alpha)$, $f$ is convex, we have $f'(\a)=\rho>0$. 
	
{Let us denote $\widetilde{L}(0,x)$ defined in \eqref{defL} by $\widetilde{L}(x)$. Note that there is no problem with integrability near $0+$ as under {\rm(B-1)} we have $L(x)\leq x^{\alpha}$.}

As an easy consequence of \eqref{Lreg} we obtain
\begin{proposition}\label{Prop1}
	Suppose that \rm{(B-1)} is satisfied. Then
	\begin{align*}
	\E B_+^\alpha \Ind{B\leq x}&=\alpha \widetilde{L}(x){-L(x)\sim \alpha\widetilde{L}(x)}
	\end{align*}
	and for any $r>0$, 
	\begin{align*}
	\E B_+^{\alpha+r}\Ind{B\leq x}&=(\alpha+r)\int_0^x t^{\alpha+r-1}\P(B>t)\dd t{-x^{\alpha+r}\P(B>x)}\sim {\frac{\alpha}{r}}x^r L(x).
	\end{align*}
\end{proposition}
\noindent Under \rm{(B-1)}, condition \rm{(B-2)} implies that  $\widetilde{L}(x)\to\infty$ as $x\to\infty$.

In this chapter the previous results in the renewal theory will be applied to the random variable $Z$ with the law defined by 
\begin{align}\label{defZ}
\P(Z\in \cdot)=\E A^{\alpha} \Ind{\log A\in \cdot}.
\end{align}

\subsection{Tails of perturbed multiplicative random walk}
In this section we study the asymptotics of $\P(R>x)$, {where $R$ is defined in \eqref{defR}.}
Under \rm{(A-2)} and \rm{(B-1)} with $\beta<\alpha$, we have $\E A^\beta<1$ and $\E B^\beta<\infty$. 
Since 
\[R_+^\beta\leq \sum_{n=1}^\infty \Pi_{n-1}^\beta (B_n)_+^\beta,\]
$R_+$ has finite moments up to $\alpha$. 
If one assumes additionally that
\begin{align}\label{ARB}
\E A^{\eta}B_+^{\alpha-\eta}<\infty\qquad\mbox{for some }\eta\in(0,\alpha),
\end{align}
then
\[\E\min\{AR,B\}_+^\alpha\leq \E (AR_+)^\alpha\Ind{AR_+\leq B_+}+\E B_+^\alpha\Ind{B_+<AR_+}<\infty,\]
because $\Ind{AR_+\leq B_+}\leq (B_+/AR_+)^{\alpha-\eta}$ and $\Ind{B_+<AR_+}\leq (AR_+/B_+)^\eta$.
The main theorem of this section is 

\begin{theorem}\label{pert}
	Assume {\rm(A-1)-(A-2)} and {\rm(B-1)-(B-2)}. 
	
	\begin{itemize}
		\item[\rm (i)]	If \eqref{ARB} holds, then
		\[x^\alpha \P(R>x)\sim\frac{\widetilde{L}(x)}{\rho}.\]
		\item[\rm (ii)]	
		If $\E A^{\alpha+\varepsilon}<\infty$ for some $\varepsilon>0$ and additionally the distribution of $Z$ defined by \eqref{defZ} is strongly non-lattice, then
		\begin{equation}\label{pert1}
		x^\alpha\P(R>x)=\frac{\widetilde{L}(x)}{\rho}-\frac{\E\min\{AR,B\}_+^\alpha}{\alpha\rho}+O(L(x)).
		\end{equation}
	\end{itemize}
\end{theorem}
\begin{remark}
	\begin{enumerate}
		\item 
		We say that the law $\mu$ is spread-out if there exists $n\in\N$ such that $n$-th convolution $\mu^{\ast n}$ has a non-zero absolutely continuous part. 
		Notice that if the law of $\log A$ is spread-out then the law of $Z$ is spread-out and so it is strongly non-lattice. If the law of $A$ has a non-trivial absolutely continuous component then the same holds for $\log A$ implying that the law of $Z$ is strongly non-lattice and we have \eqref{pert1}.
		
		\item $\E A^{\alpha+\varepsilon}<\infty$ through H\"older inequality implies
			\eqref{ARB}.
		
		\item By \eqref{pert1}, for any $\lambda\geq 1$, we have
		\[(\lambda x)^\alpha\P(R>\lambda x)-x^\alpha\P(R>x)=O(L(x)),\qquad\mbox{ as }x\to\infty,\]
		which means that $x\mapsto x^\alpha\P(R>x)\in O\Pi_L$ (see \cite[Chapter 3]{BGT89}).
	\end{enumerate}
\end{remark}
\begin{proof}
	Let $f(x)=e^{\alpha x}\P(R>e^x)$ and $\psi(x)=e^{\alpha x}\left(\P(R>e^x)-\P(AR>e^x)\right)$.
	Let $g\colon \R^n\to\R$ be a Borel function. If $(Z_k)_{k}$ is an i.i.d. sequence with the law \eqref{defZ}, then
	\[\E g(Z_1,\ldots,Z_n)=\E \Pi_n^{\alpha}g(\log A_1,\ldots, \log A_n),\]
	where $(A_k)_{k}$ are i.i.d.
	In particular,
	\begin{align}\label{EZp}
	\E Z=\E A^{\alpha}\log A\in(0,\infty).
	\end{align}
	Then, {(c.f. \eqref{reneq})}
	\begin{align}\label{renpert}
	f(x)=\psi(x)+\E A^{\alpha}f(x-\log A)=\psi(x)+\E f(x-Z).
	\end{align}
	Iterating \eqref{renpert} we obtain
	\[f(x)=\sum_{k=0}^{n-1} \E \psi(x-S_k)+\E f(x-S_{n}),\]
	where $S_n=Z_1+\ldots+ Z_n$, $S_0=0$.
	Clearly, if the law of $\log A$ given $A>0$ is non-arithmetic under $\P$, then the law of $Z$ is non-arithmetic as well.
	
	By \eqref{EZp}, the random walk $(S_n)_{n\geq 1}$ has positive drift, thus $S_n\to\infty$ with probability $1$ as $n\to\infty$.
	Moreover, since $R$ has finite moments up to $\alpha$, by Markov inequality we have $f(x)\leq \E R^{\a-\eps} e^{\eps x}$ for $\eps\in (0,\alpha)$. Thus,
	\[\E f(x-S_n)\leq \E R_+^{\a-\eps} e^{\eps x} \E e^{-\eps S_n}=\E R_+^{\a-\eps} e^{\eps x} \E \Pi_n^{\alpha-\eps}=\E R_+^{\a-\eps}e^{\eps x} \left(\E A^{\alpha-\eps}\right)^n\to 0\]
	as $n\to\infty$ and so
	\[f(x)=\sum_{k=1}^{\infty} \E \psi(x-S_k)=\int_{\R}\psi(x-z)H(\dd z),\]
	where $H$ is the renewal measure of $(S_n)_{n\geq1}$.
	
	In our case $\psi$ is not dRi (it is not even in $L^1$), so the Key Renewal Theorem is not applicable. Instead, we consider $\psi_B(x)=e^{\alpha x}\P(B>e^x)=L(e^x)$ and define $\psi_0=\psi-\psi_B$.  First we will show that $\int_{\R}\psi_0(x-z)\dd H(z)$ is convergent as $x\to\infty$ to a finite limit. Therefore, $\int_\R \psi_B(x-z)\dd H(z)$ will constitute the main part (see Theorems \ref{corL} and \ref{LTH}).
	Indeed, $\psi_0(x)=-e^{\alpha x}\P(\min\{AR,B\}>e^x)$ 
	and 
	\begin{align}\label{tech2}
	\int_{\R} e^{\alpha (x-z)}\P(\min\{AR,B\}>e^{x-z})H(\dd z)= \E \int_{(x-D,\infty)} e^{\alpha (x-z)}H(\dd z) \Ind{\min\{AR,B\}>0},
	\end{align}
	where $D=\log \min\{AR,B\}$.
	Using Fubini's theorem and changing the variable $t=z-x+D$, we obtain
	\[-\int_{\R}\psi_0(x-z)H(\dd z)=\alpha \E\min\{AR,B\}_+^\alpha \int_0^\infty e^{-\alpha t}H\left((x-D,x-D+t] \right)\dd t.\]
	By \eqref{Hineq1}, we may take the limit as $x\to\infty$ inside the integral. 
	Thus, by the Blackwell Theorem we get
	\[-\int_{\R}\psi_0(x-z)H(\dd z)\to \E\min\{AR,B\}_+^\alpha \int_0^\infty \alpha e^{-\alpha t} \frac{t}{\E Z}\dd t=\frac{\E\min\{AR,B\}_+^\alpha}{\alpha\rho}.\]
	For the main part, we have
	\[\int_\R \psi_B(x-z)H(\dd z)=\int_{(-\infty,0]} L(e^{x-z})H(\dd z)+\int_{(0,x]} L(e^{x-z})H(\dd z)+\int_{(x,\infty)} L(e^{x-z})H(\dd z).\]
	
	Let us concentrate now on the first order asymptotics, point {\rm (i)}.	
	We will show that
	\begin{align*}
	\int_{(-\infty,0]} L(e^{x-z})H(\dd z)=O(L(e^x)),\qquad \int_{(x,\infty)} L(e^{x-z})H(\dd z)=O(1)
	\end{align*}
	and 
	\[ \int_{(0,x]} L(e^{x-z})H(\dd z)\sim\frac{\widetilde{L}(x)}{\rho}\]
	Observe that $\P(Z\leq x)=\E A^\a \Ind{\log A\leq x}\leq e^{\a x}$ for any $x\in\R$ and consider the limit 
	\[\lim_{x\to\infty}\int_{(-\infty,0]} \frac{L(e^{x-z})}{L(e^x)}H(\dd z).\] 
	For any $\delta>0$, the integrand is bounded by $c e^{-\delta z}$ for some $c>1$ by Potter bounds.
	Combining this with \eqref{Hleft} and Lebesgue's Dominated Convergence Theorem we conclude that 
	\begin{align}\label{LH0}
	\int_{(-\infty,0]} L(e^{x-z})H(\dd z)\sim L(e^x)H((-\infty,0]).
	\end{align}
[Asymptotics of $H\left((-\infty,x]\right)$ which is more precise than \eqref{Hleft} is available here (see \cite{BK16}): 
$e^{\alpha x}H((-\infty,-x])\to (-\alpha\E \log A)^{-1}$ as $x\to\infty$.]

	Further, since $x^{-\alpha}L(x)=\P(B>x)\leq 1$, we have
	\[\int_{(x,\infty)} L(e^{x-z})H(\dd z)\leq \int_{(x,\infty)}e^{\alpha(x-z)}H(\dd z)=\alpha\int_0^\infty e^{-\alpha s}H\left((x,s+x]\right)\dd s\to \frac{1}{\E Z},\]
	again by the Lebesgue Dominated Convergence Theorem.
	
	The first part of the assertion will follow from Theorem \ref{corL}. Indeed, we already know that the expectation of $Z$ is strictly positive and finite. 
	Moreover, the law of $Z$ is non-arithmetic.	Thus, 
	\[\int_\R \psi_B(x-z)H(\dd z)\sim \frac{\widetilde{L}(x)}{\rho},\qquad\mbox{ as }x\to\infty.\]

	For the purpose of second order asymptotics {\rm (ii)}, we additionally assume that $\E A^{\alpha+\eps}$ is finite and that the law of $Z$ is strongly non-lattice. 
	Observe that $\E \exp(\eps Z)=\E A^{\alpha+\eps}<\infty$ and thus, the assumptions of Theorem \ref{LTH} are satisfied. Thus,
	\[\int_\R \psi_B(x-z)H(\dd z)=\frac{\widetilde{L}(x)}{\rho}+O(L(e^x)),\qquad\mbox{ as }x\to\infty.\]
	So far we have shown that
	\begin{align}\label{err}
	e^{\alpha x}\P(R>e^x)=\frac{\widetilde{L}(e^x)}{\rho}-\frac{\E\min\{AR,B\}_+^\alpha}{\alpha\rho}+O(L(e^x))+o(1),
	\end{align}
	where 
	\[
	o(1)=\int_{\R}\psi_0(x-z)H(\dd z)+\frac{\E\min\{AR,B\}_+^\alpha}{\alpha\rho}=:K(x)
	\] 
	is the error term coming from the integral of $\psi_0$. However, $L$ may be decreasing to $0$ (e.g. $L(t)\sim1/\log(t)$) and we want to be more precise here. 
	We will show that for some $\delta>0$,
	\[K(x)=o(e^{-\delta x}).\]
	and in such case we may drop $o(1)$ in \eqref{err}.
	
	Note that
	$\E A^{\alpha+\eps}<\infty$ implies $\E\min\{AR,B\}_+^{\alpha+\d}<\infty$ for $\delta<\frac{\a\eps}{\a+\eps}$. Indeed, we have
	\[\E B^{\a+\d}\Ind{0<B\leq AR}\leq \E R_+^\eta\, \E B_+^{\a+\d-\eta} A^{\eta}\leq \E R_+^\eta \left(\E B_+^{q(\a+\d-\eta)}\right)^{1/q}\left(\E A^{\eta p}\right)^{1/p},\]
	where $p^{-1}+q^{-1}=1$ and $\eta>0$. The right hand side is finite for $\eta\in(\d\frac{\a+\eps}{\eps},\a)$ with $p=\frac{\a+\eps}{\eta}$. Analogously we show that $\E (AR_+)^{\a+\d}\Ind{B> AR}<\infty$.
	We write (recall that $D=\log \min\{AR,B\}$)
	\begin{align*}
	K(x)=&-\alpha \E\min\{AR,B\}_+^\alpha \int_0^\infty e^{-\alpha t}\left(H\left((x-D,x-D+t]\right)-\frac{t}{\E Z}\right)\dd t\, \Ind{D\leq x}\\
	&-\alpha \E\min\{AR,B\}_+^\alpha \int_0^\infty e^{-\alpha t}H\left((x-D,x-D+t]\right)\dd t\, \Ind{D>x} \\\
	& +\frac{\E\min\{AR,B\}_+^\alpha \Ind{D>x}}{\a \E Z} = K_1+K_2+K_3.
	\end{align*}
	We have
	\[|K_2+K_3|\leq C \E \min\{AR,B\}_+^\alpha \Ind{\min\{AR,B\}>e^x}\leq c e^{-\d x}\min\{AR,B\}^{\a+\d}.\]
	Moreover, under our setup we know that for $R(x)=H((-\infty,x])-x/\E Z$ one has $R(x)-(2(\E Z)^2)^{-1}\E Z^2=o(\exp(-rx))$ as $x\to\infty$ and thus $|R(x)-(2(\E Z)^2)^{-1}\E Z^2|\leq C \exp(-rx)$ for some $C>0$ and $0<r<\delta$ and all $x\geq0$. 
	Then
	\begin{align*}
	|K_1| & \leq \alpha \E\min\{AR,B\}_+^\alpha \int_0^\infty e^{-\alpha t}\left(|R(x-D+t)|+|R(x-D)|\right)\dd t \Ind{D\leq x} \\
	& \leq \widetilde{C} e^{-r x} \E\min\{AR,B\}_+^{\alpha+r}
	\end{align*}
	and the conclusion follows.

\end{proof}

\section{Proofs}\label{proofs}

\begin{proof}[Proof of Theorem~\ref{corL}]
	Using the definition of a slowly varying function, it is easy to see that the integral over $(0,x_0]$ is $O(L(\exp(x)))$ as $x\to\infty$.
	Indeed, for $x_0>0$,
	\[\limsup_{x\to\infty}\int_{(0,x_0]} \frac{L(e^{x-z})}{L(e^x)}H(\dd z) \leq \lim_{x\to\infty} \sup_{z\in(0,x_0]}\frac{L(e^{x-z})}{L(e^x)}  H((0,x_0])=H((0,x_0]),\]
	since the convergence in \eqref{reg} is locally uniform outside zero (\cite[Theorem 1.5.2]{BGT89}).
	Moreover, the integral over $(x-x_0,x]$ is $O(1)$ as $x\to\infty$. Indeed, by the local boundedness of $L$ we have
	\[\int_{(x-x_0,x]} L(e^{x-z})H(\dd z)\leq \sup_{t\in [0,x_0)}L(e^t)\, H\left((x-x_0,x]\right)\]
	and, by the Blackwell theorem, the right hand side above converges.
	Thus, it is enough to concentrate on the integral over $(x_0,x-x_0]$. 
	Let us fix $n\in\mathbb{N}$ and $\eps>0$ and observe that
	\[\bigcup_{k=0}^{\left\lfloor n(x-2 x_0)\right\rfloor-1} \left(x_0+\frac{k}{n},x_0+\frac{k+1}{n}\right]\subset(x_0,x-x_0]\subset\bigcup_{k=0}^{\left\lceil n(x-2 x_0)\right\rceil-1} \left(x_0+\frac{k}{n},x_0+\frac{k+1}{n}\right].\]
	Further, by Potter bounds (\cite[Theorem~1.5.6 (i)]{BGT89}), if $z\in \left(x_0+k/n,x_0+(k+1)/n\right]$, for any $\delta>0$ and sufficiently large $x$ and $x_0$ we have
	\[\frac{L(e^{x-z})}{ L(e^{x-x_0-k/n})}\leq (1+\eps)e^{\delta/n}\leq(1+\eps)^2\]
	and similarly for the lower bound.
	Moreover, let $x_0$ be such that for any $k\geq0$,
	\[(1-\eps)\frac{1}{n \E Z} \leq H\left(\left(x_0+\frac{k}{n},x_0+\frac{k+1}{n}\right]\right)\leq (1+\eps)\frac{1}{n \E Z}.\]
	Altogether, above considerations yield 
	\begin{multline*}
	\frac{(1-\eps)^3}{\E Z} \frac{1}{n}\sum_{k=0}^{\left\lfloor n(x-2 x_0)\right\rfloor-1} L(e^{x-x_0-k/n}) \leq \int_{(x_0,x-x_0]} L(e^{x-z})\dd H(z) \\ \leq \frac{(1+\eps)^3}{\E Z} \frac{1}{n} \sum_{k=0}^{\left\lceil n(x-2 x_0)\right\rceil-1} L(e^{x-x_0-k/n})
	\end{multline*}
	for any $n$, $\eps$ and sufficiently large $x_0$. This gives us that 
	\[\int_{(x_0,x-x_0]} L(e^{x-z})H(\dd z)\sim \frac{1}{\E Z} \int_{x_0}^{x-x_0} L(e^{t})\dd t\sim \frac{1}{\E Z} \widetilde{L}(x_0,e^x)\]
	and the assertion follows.
\end{proof}

\begin{proof}[Proof of Theorem~\ref{LTH}]
	We have
	\begin{align*}\int_\R L(e^{x-z})H(\dd z)&=\int_{(-\infty,0]}  L(e^{x-z})H(\dd z)+\int_{(0,\infty)} e^{\alpha(x-z)}\P(B>e^{x})\dd H(z) \\
	&+\int_{(0,\infty)} e^{\alpha(x-z)}\P(e^x\geq B>e^{x-z})\dd H(z).\end{align*}
	We already know that the first term is asymptotically equivalent to $L(\exp(x))H((-\infty,0])$ (see \eqref{LH0}).
	
	The second term equals $L(\exp(x))\int_0^\infty \exp(-\alpha z)H(\dd z)$ and the integral is convergent, thus it is of the same order as the first one.
	
	The main contribution comes from the third term, which is equal to
	\[(\E Z)^{-1}\int_0^\infty e^{\alpha(x-z)}\P(e^x\geq B>e^{x-z})\dd z+\E \int_{(x-\log B,\infty)} e^{\alpha(x-z)}\dd R(z)\Ind{0<B\leq e^x}=K_1+K_2,\]
	where $R(z)=H((-\infty,z])-z/\E Z$. Since $\int_0^\infty L(\exp(x-z))\dd z=\widetilde{L}(0,\exp(x))$, we have
	\[K_1=\frac{1}{\E Z}\widetilde{L}(0,e^x)+O(L(e^x))\]
	and after integrating by parts and changing the variable $t=z-x+\log B$,
	\[|K_2|\leq \alpha\E B^\alpha I(0<B\leq e^x)\int_0^\infty e^{-\alpha t}|(R(t+x-\log B)-R(x-\log B)|\dd t.\]
	It remains to show that $K_2(x)=O(L(e^x))$. 
	 
	Since $\E \exp(\varepsilon Z)<\infty$, we get that $\exp(\eps x)\P(Z>x)\to 0$. Moreover, by assumption, the distribution of $Z$ is strongly non-lattice. Thus, by \eqref{Hright}, there exists $r>0$ such that $R(x)-C=o(\exp(-r x))$ as $x\to\infty$, where $C=(2\E Z)^{-1}\E Z^2$. This implies that $|R(x)-C|\leq K \exp(-r x)$ for all $x>0$ and some finite $K$.
	
	We have
	\begin{align*}
	|K_2|\leq &\alpha \E B^\alpha I(0<B\leq e^x)\int_0^\infty e^{-\alpha t}|(R(t+x-\log B)-C|\dd t \\
	&+ \E B^\alpha I(0<B\leq e^x)|R(x-\log B)-C| \\
	&\leq K\left(1+\frac{\alpha}{\alpha+r}\right) e^{-r x}\E B_+^{\alpha+r}\Ind{B\leq e^x}=O(L(e^x))
	\end{align*}
	by Proposition \ref{Prop1}.
\end{proof}

\subsection*{Acknowledgements}
The authors are thankful to anonymous referee for simplifying the proof of Theorem \ref{corL}. Remark \ref{rem} was proposed by a referee. Ewa Damek was partially supported by the NCN Grant UMO-2014/15/B/ST1/00060.
Bartosz Ko{\l}odziejek was partially supported by the NCN Grant UMO-2015/19/D/ST1/03107.


\begin{thebibliography}{99}

\bibitem{Ara06}
V.~F. Araman and P.~W. Glynn.
\newblock Tail asymptotics for the maximum of perturbed random walk.
\newblock \emph{Ann. Appl. Probab.}, 16\penalty0 (3):\penalty0 1411--1431,
2006.

\bibitem{AMN78}
K.~B. Athreya, D.~McDonald and P.~Ney.
\newblock Limit theorems for semi-{M}arkov processes and renewal theory for {M}arkov chains.
\newblock \emph{Ann. Probab.}, 6\penalty0 (5):\penalty0 788--797,
1978.

\bibitem{BGT89}
N.~H. Bingham, C.~M. Goldie, and J.~L. Teugels.
\newblock \emph{Regular variation}, volume~27 of \emph{Encyclopedia of
	Mathematics and its Applications}.
\newblock Cambridge University Press, Cambridge, 1989.

\bibitem{Black53}
D.~Blackwell.
\newblock Extension of a renewal theorem.
\newblock \emph{Pacific J. Math.}, 3:\penalty0 315--320, 1953.

\bibitem{BDM2016}
D.~Buraczewski, E.~Damek, and T. Mikosch.
\newblock \emph{Stochastic Models with Power-Law Tails. The Equation $X=AX+B$.}
\newblock Springer Series in Operations Research and Financial Engineering.
Springer International Publishing, Switzerland, 2016.

\bibitem{Gol91}
C.~M. Goldie.
\newblock Implicit renewal theory and tails of solutions of random equations.
\newblock \emph{Ann. Appl. Probab.}, 1\penalty0 (1):\penalty0 126--166, 1991.

\bibitem{Ren11}
P.~Hitczenko and J.~Weso{\l}owski.
\newblock Renorming divergent perpetuities.
\newblock \emph{Bernoulli}, 17\penalty0 (3):\penalty0 880--894, 2011. 

\bibitem{Iks}
A.~Iksanov.
\newblock \emph{Renewal theory for perturbed random walks and similar
	processes.}
\newblock Probability and its Applications. 
Birkh\"auser/Springer, Cham, 2017.

\bibitem{IksPi14}
A.~Iksanov and A.~Pilipenko.
\newblock On the maximum of a perturbed random walk.
\newblock \emph{Statist. Probab. Lett.}, 92:\penalty0 168--172, 2014.

\bibitem{IksPiSa17}
A.~Iksanov, A.~Pilipenko and I.~Samoilenko.
\newblock Functional limit theorems for the maxima of perturbed random walk and divergent perpetuities in the {$M_1$}-topology
\newblock \emph{Extremes}, 20\penalty0(3):\penalty0 567--583, 2017.

\bibitem{KM96}
H.~Kesten and R.~A. Maller.
\newblock Two renewal theorems for general random walks tending to
infinity.
\newblock \emph{Probab. Theory Related Fields}, 106\penalty0 (1): 1--38,
1996.

\bibitem{Kevei2016}
P.~Kevei.
\newblock A note on the {K}esten-{G}rincevi\v cius-{G}oldie theorem.
\newblock \emph{Electron. Commun. Probab.}, 21\penalty0 (51):\penalty0 1--12,
2016.

\bibitem{BK16}
B.~Ko{\l}odziejek.
\newblock The left tail of renewal measure.
\newblock \emph{Statist. Probab. Lett.}, 129:\penalty0 306--310, 2017.

\bibitem{Palm07}
Z.~Palmowski and B.~Zwart.
\newblock Tail asymptotics of the supremum of a regenerative process.
\newblock \emph{J. Appl. Probab.}, 44\penalty0 (2):\penalty0 349--365, 2007.

\bibitem{Schle98}
S.~Schlegel.
\newblock Ruin probabilities in perturbed risk models
\newblock \emph {Insurance Math. Econom.}, 22\penalty0: 93-104, 1998. 

\bibitem{Schm95} 
H.~Schmidli.
\newblock Cram\'er-Lundberg approximations for ruin probabilities of risk processes perturbed by diffusion 
\newblock \emph {Insurance Math. Econom.}, 16\penalty0: 135-149, 1995. 

\bibitem{Stone65}
C.~Stone.
\newblock On moment generating functions and renewal theory.
\newblock \emph{Ann. Math. Statist.}, 36:\penalty0 1298--1301, 1965.

\bibitem{Wan14}
Y.~Wang.
\newblock Convergence to the maximum process of a fractional {B}rownian motion
with shot noise.
\newblock \emph{Statist. Probab. Lett.}, 90:\penalty0 33--41, 2014.

\end{thebibliography}
\end{document}